\documentclass{amsart}
\pdfoutput=1
\usepackage{microtype}
\usepackage{amsmath,amssymb,amsthm}
\usepackage{url}
\usepackage{doi}

\usepackage{amsrefs}

\newcommand{\FF}{\ensuremath{\mathbb{F}}}
\newcommand{\ZZ}{\ensuremath{\mathbb{Z}}}
\newcommand{\QQ}{\ensuremath{\mathbb{Q}}}
\newcommand{\RR}{\ensuremath{\mathbb{R}}}
\newcommand{\CC}{\ensuremath{\mathbb{C}}}
\newcommand{\OO}{\ensuremath{\mathcal{O}}}
\newcommand{\XX}{\ensuremath{\mathcal{X}}}
\newcommand{\LL}{\ensuremath{\mathcal{L}}}
\newcommand{\MM}{\ensuremath{\mathcal{M}}}
\newcommand{\PP}{\ensuremath{\mathcal{P}}}
\newcommand{\supp}{\ensuremath{\operatorname{supp}}}

\renewcommand{\Re}{\ensuremath{\operatorname{Re}}}
\renewcommand{\Im}{\ensuremath{\operatorname{Im}}}

\newtheorem{theorem}{Theorem}[section]
\newtheorem{lemma}[theorem]{Lemma}

\theoremstyle{definition}

\numberwithin{equation}{section}

\begin{document}

\title[Ideal forms\dots]{Ideal forms of Coppersmith's theorem\\
and Guruswami-Sudan list decoding}

\author{Henry Cohn}
\address{Microsoft Research New England\\
One Memorial Drive\\
Cambridge, MA 02142} \email{cohn@microsoft.com}

\author{Nadia Heninger}
\address{Department of Computer and Information Science\\
University of Pennsylvania\\
Philadelphia, PA 19104}
\email{nadiah@cs.princeton.edu}

\thanks{An extended abstract of this paper was published
in the Proceedings of the Second Symposium on Innovations in Computer Science
(Beijing, January 7--9, 2011), Tsinghua University Press, pages 298--308.
N.H.~was supported by an internship at Microsoft Research New England and an
NSF Graduate Research Fellowship.}

\date{June 25, 2013}

\maketitle

\begin{abstract}
We develop a framework for solving polynomial equations with size
constraints on solutions. We obtain our results by showing how to apply
a technique of Coppersmith for finding small solutions of polynomial
equations modulo integers to analogous problems over polynomial rings,
number fields, and function fields. This gives us a unified view of
several problems arising naturally in cryptography, coding theory, and
the study of lattices.  We give (1) a polynomial-time algorithm for
finding small solutions of polynomial equations modulo ideals over
algebraic number fields, (2) a faster variant of the Guruswami-Sudan
algorithm for list decoding of Reed-Solomon codes, and (3) an algorithm
for list decoding of algebraic-geometric codes that handles both
single-point and multi-point codes.  Coppersmith's algorithm uses
lattice basis reduction to find a short vector in a carefully
constructed lattice; powerful analogies from algebraic number theory
allow us to identify the appropriate analogue of a lattice in each
application and provide efficient algorithms to find a suitably short
vector, thus allowing us to give completely parallel proofs of the
above theorems.
\end{abstract}

\section{Introduction}

Many important problems in areas ranging from cryptanalysis to coding theory
amount to solving polynomial equations with side constraints or partial
information about the solutions. One of the most important cases is solving
equations given size bounds on the solutions. Coppersmith's algorithm is a
celebrated technique for finding small solutions to polynomial equations
modulo integers, and it has many important applications in cryptography,
particularly in the cryptanalysis of the RSA cryptosystem.

In this paper, we show how the ideas of Coppersmith's theorem can be
extended to a more general framework encompassing the original
number-theoretic problem, list decoding of Reed-Solomon and
algebraic-geometric codes, and the problem of finding solutions to
polynomial equations modulo ideals in rings of algebraic integers.
These seemingly different problems are all perfectly analogous when
viewed from the perspective of algebraic number theory.

Coppersmith's algorithm provides a key example of the power of lattice basis
reduction. To extend the method beyond the integers, we examine the analogous
structures for polynomial rings, number fields, and function fields. Ideals
over number fields have a natural embedding into a lattice, and thus we can
find a short vector simply by applying the LLL algorithm to this canonical
embedding. In contrast to integer lattices, it turns out that lattice basis
reduction is much easier over a lattice of polynomials, and in fact a
shortest vector can always be found in polynomial time.  Recasting the list
decoding problem in this framework allows us to take advantage of very
efficient reduction algorithms and thus achieve the fastest known list
decoding algorithm for Reed-Solomon codes.

To extend this approach to function fields, we must overcome certain
technical difficulties.  Along the way, we prove a more general result
about finding short vectors under arbitrary non-Archimedean norms, which may
have further applications beyond list decoding of algebraic-geometric codes.
As an illustration of the generality of our approach, we give the first list
decoding algorithm that works for all algebraic-geometric codes, not just
those defined using a single-point divisor.

In the remainder of the introduction, we set up our framework with a
brief review of Coppersmith's theorem, and then state our theorems on
polynomial rings, number fields, and function fields.

\subsection{Coppersmith's theorem}

The following extension of Coppersmith's theorem \cite{coppersmith} was
developed by Howgrave-Graham \cite{howgrave-graham-gcd} and May
\cite{may:thesis}.

\begin{theorem}[\cite{coppersmith,howgrave-graham-gcd,may:thesis}]
\label{theorem:coppersmith} Let $f(x)$ be a monic polynomial of degree
$d$ with coefficients modulo an integer $N>1$, and suppose $0 < \beta
\le 1$. In time polynomial in $\log N$ and $d$, one can find all
integers $w$ such that
\[|w| \leq N^{\beta^2/d}\]
and
\[\gcd(f(w),N) \ge N^\beta.\]
\end{theorem}

Note that when $\beta=1$, this amounts to finding all sufficiently small
solutions of $f(w) \equiv 0 \pmod{N}$, and the general theorem amounts to
solving $f(w) \equiv 0 \pmod{B}$, where $B$ is a large, unknown factor of
$N$.

We give a brief example to illustrate the power of this theorem in
cryptography \cite{coppersmith,howgrave-graham-gcd}. Imagine that an
adversary has obtained through a side-channel attack some knowledge
about one of the prime factors $p$ of an RSA modulus $N = pq$, for
example some of its most significant bits. We denote this known
quantity by $r$. Then we may write $p = r+w$, where the bound on $w$
depends on how many bits of $p$ are known.  Suppose more than half of
the bits have leaked, i.e., $0 \le w \le N^{1/4 - o(1)}$ (we assume, as
is typical, that $p$ and $q$ are both $N^{1/2+o(1)}$). Now let $f(x) =
x+r$ and $\beta = 1/2 + o(1)$. Theorem~\ref{theorem:coppersmith} tells
us that we can in polynomial time learn $w$, and hence $p$, thereby
factoring $N$.

Further applications of this theorem in cryptography include other
partial key recovery attacks against RSA \cite{boneh:rsa, blomer:rsa},
attacks on stereotyped messages and improper padding
\cite{coppersmith}, and the proof of security for the RSA-OAEP+ padding
scheme \cite{shoup:oaep}.  See \cite{may:survey} for many other
applications.

It is remarkable that Theorem~\ref{theorem:coppersmith} allows us to
solve polynomial equations modulo $N$ without knowing the factorization
of $N$, and this fact is critical for the cryptanalytic applications.
However, even if one already has the factorization,
Theorem~\ref{theorem:coppersmith} remains nontrivial if $N$ has many
prime factors.
To solve an equation modulo a composite number, one generally solves
the equation modulo each prime power factor of the modulus and uses the
Chinese remainder theorem to construct solutions for the original
modulus. (Recall that modulo a prime, such equations can be solved in
polynomial time, and we can use Hensel's lemma to lift the solutions to
prime power moduli.) The number of possible solutions can be
exponential in the number of prime factors, in which case it is
infeasible to enumerate all of the roots and then select those that are
within the desired range. In fact, the problem of determining whether
there is a root in an arbitrary given interval is NP-complete
\cite{manders}.  Of course, if $N$ has only two prime factors, then
there can be only $d^2$ solutions modulo $N$, but our methods are
incapable of distinguishing between numbers with two or many prime
factors.

It is not even obvious that the number of roots modulo $N$ of size at most
$N^{1/d}$ is polynomially bounded in terms of the number of digits of $N$.
From this perspective, the exponent $1/d$ is optimal without further
assumptions, because $f(x)=x^d$ will have exponentially many roots modulo
$N=k^d$ of absolute value at most $N^{1/d+\varepsilon}$ (specifically, the
$2N^{\varepsilon}$ such multiples of $k$). Theorem~\ref{theorem:coppersmith}
can be seen as a constructive bound on the number of solutions. See
\cite{coppersmith:small} for further discussion of this argument and
\cite{konyagin} for non-constructive bounds.

\subsection{A polynomial analogue}

To introduce our analogies, we will begin with the simplest and most
familiar case: polynomials.

There is an important analogy in number theory between the ring $\ZZ$
of integers and the ring $F[z]$ of univariate polynomials over a field
$F$.  To formulate the analogue of Coppersmith's theorem, one just
needs to recognize that the degree of a polynomial is the appropriate
measure of its size. Thus, the polynomial version of Coppersmith's
theorem should involve finding low-degree solutions of polynomial
equations over $F[z]$ modulo a polynomial $p(z)$.  That is, given a
polynomial $f(x) = \sum_{i=0}^d f_i(z) x^i$ with coefficients $f_i(z)
\in F[z]$, we seek low-degree polynomials $w(z) \in F[z]$ such that
$f(w(z)) \equiv 0 \pmod{p(z)}$.

In the following theorem, we assume that we can efficiently represent
and manipulate elements of $F$, and that we can find roots in $F[z]$ of
polynomials over $F[z]$.  For example, that holds if we can factor
bivariate polynomials over $F$ in polynomial time.  This assumption
holds for many fields, including $\QQ$ and even number fields
\cite{lenstra:multi} as well as all finite fields \cite{vzGK} (with a
randomized algorithm in the latter case).

\begin{theorem}\label{theorem:polycop}
Let $f(x)$ be a monic polynomial in $x$ of degree $d$ over $F[z]$ with
coefficients modulo $p(z)$, where $\deg_z p(z) = n > 0$. In polynomial time,
for $0 < \beta \le 1$, we can find all $w(z) \in F[z]$ such that
\[
\deg_z w(z) < \beta^2n/d
\]
and
\[
\deg_z \gcd(f(w(z)),p(z)) \ge \beta n.
\]
\end{theorem}

In the case when $p(z)$ factors completely into linear factors, this
theorem is equivalent to the influential Guruswami-Sudan theorem on
list decoding of Reed-Solomon codes \cite{guruswami-sudan}.  See
Section~\ref{section:rsld} for the details of the equivalence.  The
above statement of Theorem~\ref{theorem:polycop}, as well as the
extension to higher-degree irreducible factors, appear to be new.

It has long been recognized that the Coppersmith and Guruswami-Sudan
theorems are in some way analogous, although we are unaware of any
previous, comparably explicit statement of the analogy. Boneh used
Coppersmith's theorem in work on Chinese remainder theorem codes
inspired by the Guruswami-Sudan theorem \cite{boneh:smooth}, and in a
brief aside in the middle of \cite{bernstein}, Bernstein noted that the
Guruswami-Sudan theorem is the polynomial analogue of a related theorem
of Coppersmith, Howgrave-Graham, and Nagaraj \cite{chgn}.
Alekhnovich~\cite{alekhnovich} formulated the problem of list-decoding of
Reed-Solomon codes in terms of finding a Gr\"obner basis for a polynomial ideal,
and he gave an algorithm for finding a short vector in a polynomial lattice to do so.
See also
\cite{guruswami-sahai-sudan} for a general ideal-theoretic setting for
coding theory, and \cite{sudan} for a survey of relationships between
list decoding and number-theoretic codes.

\subsection{Number fields}

A number field is a finite extension of the field $\QQ$ of rational numbers.
Thus it is natural to investigate how a statement over the rationals, the
simplest number field, extends to more general number fields.  We extend our
analogy by adapting Coppersmith's theorem to the number field case.

Every number field $K$ is of the form
$$
K = \QQ(\alpha) = \{a_0 + a_1\alpha + \dots + a_{n-1}\alpha^{n-1}
: a_0,\dots,a_{n-1} \in \QQ\},
$$
where $\alpha$ is an algebraic number of degree $n$ (i.e., a root of an
irreducible polynomial of degree $n$ over $\QQ$). The degree of $K$ is
defined to be $n$.  Within $K$, there is a ring $\OO_K$ called the ring of
algebraic integers in $K$.  It plays the same role within the field $K$ as
the ring $\ZZ$ of integers plays within $\QQ$. Sometimes $\OO_K$ is of the
form $\ZZ[\alpha]$, but sometimes it does not even have a single generator.

Recall that an \emph{ideal} in a ring is a non-empty subset closed
under addition and under multiplication by arbitrary elements of the
ring. (Intuitively, it is a subset modulo which one can reduce elements
of the ring.)  For example, the multiples of any fixed element form an
ideal, called a \emph{principal ideal}.  In $\ZZ$ every ideal is of
that form, but that is not usually true in $\OO_K$.

In $\OO_K$, we study the solutions of polynomial equations modulo
ideals, the analogue of such equations modulo integers in $\ZZ$. To
measure the size of a nonzero ideal $I$ in $\OO_K$, we will use its
norm $N(I) = |\OO_K/I|$, i.e., the size of the quotient ring.

A final conceptual issue that makes this case more subtle is that a
number field of degree $n$ has $n$ absolute values $|\cdot|_i$
corresponding to its $n$ embeddings into $\CC$ (as we will explain in
Section~\ref{section:numberfield}), and to obtain the theorem it is
necessary to bound them all simultaneously.

The number field analogue of Coppersmith's theorem is as follows:

\begin{theorem}
\label{theorem:nf-coppersmith} Let $K$ be a number field of degree $n$
with ring of integers $\OO_K$, $f(x) \in \OO_K[x]$ a monic polynomial
of degree $d$, and $I \subsetneq \OO_K$ an ideal in $\OO_K$. Assume
that we are given $\OO_K$ and $I$ explicitly by integral bases. For $0
< \beta \le 1$ and $\lambda_1,\dots,\lambda_n>0$, in time polynomial in
the input length and exponential in $n^2$ we can find all $w \in \OO_K$
with $|w|_i < \lambda_i$ such that
\[
N(\gcd(f(w)\OO_K,I)) > N(I)^\beta,
\]
provided that
\[
\prod_i \lambda_i < N(I)^{\beta^2/d}.
\]
Furthermore, in polynomial time we can find all such $w$ provided that
\[
\prod_i \lambda_i < (2+o(1))^{-n^2 / 2}  N(I)^{\beta^2/d}.
\]
\end{theorem}

Equivalently, we can find small solutions of equations $f(x) \equiv 0
\pmod{J}$, where the ideal $J$ is a large divisor of $I$. Using improved
lattice basis reduction algorithms \cite{AKS} we can achieve a running time
that is slightly subexponential in $n^2$.  Note also that $\gcd(f(w)\OO_K,I)$
is the largest ideal that contains both the principal ideal $f(w)\OO_K$ and
$I$; in other words, it is their sum $f(w)\OO_K + I$.

When $n$ is fixed, our algorithm runs in polynomial time, but the
dependence on $n$ is exponential.  That appears to be unavoidable using
our techniques, but it is not a serious drawback.  Many
number-theoretic algorithms behave poorly for high-degree number
fields, and most computations are therefore done in low-degree cases.
Even for a fixed number field $K$, Theorem~\ref{theorem:nf-coppersmith}
remains of interest.

Similar ideas to Theorem~\ref{theorem:nf-coppersmith} were developed independently
by Coxon for list decoding of
number field codes \cite{coxon}.  His algorithm is more flexible than our theorem
in allowing weighted list decoding, but he does not develop an analogue of
Coppersmith's theorem.  Similar results were also achieved by Biasse and Quintin~\cite{biasse-quintin}.

Several problems over number fields have been proposed as the basis for
cryptosystems; see, for example, \cite{nfcrypto} for a survey of problems
over quadratic number fields.  More recently, Peikert and Rosen
\cite{peikert} and Lyubashevsky, Peikert, and Regev \cite{LPR} developed
lattice-based cryptographic schemes using lattices representing the canonical
embeddings of ideals in number fields. As a special case,
Theorem~\ref{theorem:nf-coppersmith} can be used to solve certain cases of
the bounded-distance decoding problem for such lattices, and improving our
approximation factor from $(2+o(1))^{-n^2/2}$ to $2^{-n} \sqrt{|\Delta_K|}$,
where $\Delta_K$ is the discriminant of $K$, would solve the problem in
general; see Section~\ref{subsec:bdd} for more details.

In addition, number fields have many applications to purely classical
problems, the most prominent example being the number field sieve
factoring algorithm.  All sieve algorithms require generating smooth
numbers, and in this context Boneh \cite{boneh:smooth} showed how to
use Coppersmith's theorem to find smooth integer solutions of
polynomials in short intervals.  Using
Theorem~\ref{theorem:nf-coppersmith} analogously, one can do the same
over number fields.

We prove Theorem~\ref{theorem:nf-coppersmith} in
Section~\ref{section:numberfield}.

\subsection{Function fields}
\label{subsec:ffsintro}

Algebraic number theorists have developed a more sophisticated version of the
analogy between the ring of integers and polynomial rings.  In this analogy,
the analogues of number fields are called function fields; they are the
fields of rational functions on algebraic curves over finite fields. The
parallels between number fields and function fields are truly astonishing,
and this analogy has played a crucial role in the development of number
theory over the last century.

We now complete our analogy in this paper by extending Coppersmith's theorem to the
function field case.  See Section~\ref{section:functionfield} for a
detailed review of the setting and notation.

\begin{theorem}
\label{theorem:ff-coppersmith} Let $\XX$ be a smooth, projective, absolutely
irreducible algebraic curve over $\FF_q$, and let $K$ be its function field
over $\FF_q$.  Let $D$ be a divisor on $\XX$ whose support $\supp(D)$ is
contained in the $\FF_q$-rational points $\XX(\FF_q)$, let $S$ be a subset of
$\XX(\FF_q)$ that properly contains $\supp(D)$, let $\OO_S$ be the subring of
$K$ consisting of functions with poles only in $S$, and let $\LL(D)$ be the
Riemann-Roch space
\[
\LL(D) = \{0\} \cup \{f \in K^* : (f)+D \succeq 0\}.
\]
Let $f(x) \in \OO_S[x]$ be a monic polynomial of degree $d$, and let
$I$ be a proper ideal in $\OO_S$.

Then in probabilistic polynomial time, we can find all $w \in \LL(D)$
such that
\[
N(\gcd(f(w)\OO_S,I)) \ge N(I)^\beta,
\]
provided that
\[
q^{\deg(D)} < N(I)^{\beta^2/d}.
\]
\end{theorem}

In the case when $S$ contains only a single point, the function field version
of Coppersmith's theorem is equivalent to the Guruswami-Sudan theorem on
list-decoding of algebraic-geometric codes, as we will outline in
Section~\ref{section:functionfield}.  The Guruswami-Sudan theorem and the
earlier Shokrollahi-Wasserman theorem \cite{shokrollahi-wasserman} are
specialized to that case, which covers many but not all algebraic-geometric
codes. Our theorem extends list decoding to the full range of such codes.

We assume that we can efficiently compute bases of Riemann-Roch spaces
for divisors in $\XX$.  That can be done in many important cases (for
example, for a smooth plane curve, or even one with ordinary multiple
points \cite{huangierardi}), and it is a reasonable assumption because
even the encoding problem for algebraic-geometric codes requires a
basis of a Riemann-Roch space.  Note also that although our algorithm
is probabilistic, it is guaranteed to give the correct solution in
expected polynomial time; in other words, it is a ``Las Vegas''
algorithm.

Alekhnovich's polynomial module approach to list decoding has been adapted to
a special class of algebraic-geometric codes by Beelen and Brander \cite{beelen-brander2}.
Their work is thus closer in spirit to ours than the Guruswami-Sudan
and Shokrollahi-Wasserman papers are, but they focus on the process of
interpolation and view modules as a means to compute interpolating polynomials.

We prove Theorem~\ref{theorem:ff-coppersmith} in
Section~\ref{section:functionfield}.

\subsection{Analogies in number theory}

The connections we have described are not isolated phenomena. Many
theorems in number theory and algebraic geometry have parallel versions
for the integers and for polynomial rings, or more generally for number
fields and function fields, and translating statements or techniques
between these settings can lead to valuable insights.

One particular advantage of this sort of arbitrage is that proving
results for polynomial rings is usually easier.  For example, the prime
number theorem for $\ZZ$ is a deep theorem, but the analogue for the
polynomial ring $\FF_q[z]$ over a finite field is much simpler.  It
says that asymptotically a $1/n$ fraction of the $q^n$ monic
polynomials of degree $n$ are irreducible, and in fact the error term
is on the order of $q^{n/2}$ (see Lemma~14.38 in \cite{vzgg}). Proving
a similarly strong version of the prime number theorem for $\ZZ$ would
amount to proving the Riemann hypothesis.  Similarly, the ABC
conjecture for $\ZZ$ is a profound unsolved problem, while for
polynomials rings it has an elementary proof \cite{mason}.

Thus, polynomial rings are worlds in which many of the fondest dreams
of mathematicians have come true.  If a result cannot be proved in such
a setting, then it is probably not even worth trying to prove it in
$\ZZ$. If it can be proved for polynomial rings, then the techniques
may not apply to the integers, but they often provide inspiration for
how a proof might work if technical obstacles can be overcome.

Similarly, in computer science many computational problems that appear to be
difficult for integers are tractable for polynomials. For example, factoring
polynomials can be done in polynomial time for many fields, while for
integers the problem seems to be hard.  The polynomial analogue of the
shortest vector problem for lattices can be solved exactly in polynomial time
\cite{vonzurgathen}, while for integer lattices the problem is NP-hard
\cite{ajtai}. This difference in the difficulty of lattice problems is at the
root of the poor running time in Theorem~\ref{theorem:nf-coppersmith} for
number fields of high degree.

The analogies that we develop here between cryptanalysis and coding theory extend
further.  For example, multivariate versions of Coppersmith's theorem
correspond to list decoding of Parvaresh-Vardy and Guruswami-Rudra codes \cite{cohn-heninger}.

\section{Preliminaries}

One of the main steps in Coppersmith's theorem uses lattice basis reduction
to find a short vector in a lattice. In this section, we review preliminaries
on integral lattices and introduce the analogues that we will use in our
generalizations.

\subsection{Integer lattices}

Recall that a \emph{lattice} in $\RR^m$ is a discrete subgroup of rank
$m$. Equivalently, it is the set of integer linear combinations of a
basis of $\RR^m$.

The \emph{determinant} $\det(L)$ of a lattice $L$ is the absolute value
of the determinant of any basis matrix; it is not difficult to show
that it is independent of the choice of basis.  One way to see why is
that the determinant is the volume of the quotient $\RR^m/L$, or
equivalently the volume of a fundamental parallelotope.

One of the fundamental problems in lattice theory is finding short
vectors in lattices, with respect to the $\ell_p$ norm
\[
|v|_p = \left(\sum_{i=1}^m |v_i|^p\right)^{1/p}.
\]
Most often we use the $\ell_2$ norm, which is of course the usual
Euclidean distance.  The LLL lattice basis reduction algorithm
\cite{lll} can be used to find a short vector in a lattice.

\begin{theorem}[\cite{lll}]
Given a basis of a lattice $L$ in $\QQ^m$, a nonzero vector $v \in L$
satisfying
\[
|v|_2 \leq 2^{(m-1)/4} \det(L)^{1/m}
\]
can be found in polynomial time.
\end{theorem}

Note that the LLL algorithm's input is a rational lattice, and the
rationality plays an important role in the running time analysis. In the
proof of Theorem~\ref{theorem:nf-coppersmith}, we must apply it to a lattice
whose basis vectors are not in $\QQ^m$; however, for our purposes using a
close rational approximation suffices.  Specifically, we simply approximate
the given basis $b_1,\dots,b_m$ with rational vectors $b_1',\dots,b_m'$. It
is easy to check that using a polynomial number of digits suffices to
approximate the determinant.  (The number of digits depends polynomially on
the dimension $m$ and the logarithmic sizes of the entries in the basis
vectors.) Then we find a short vector $\sum c_i b'_i$ with $c_i \in \ZZ$,
where the coefficients $c_i$ have only polynomially many digits because they
are the output of the polynomial-time algorithm LLL algorithm. If $b'_i$
approximates $b_i$ to enough digits, then $\sum c_i b_i$ will have
essentially the same length, and again a polynomial number of digits
suffices.  Strictly speaking, this process makes the approximation factor
slightly worse, but the difference is insignificant, and we could use a
better version of the LLL algorithm to achieve the same $2^{(m-1)/4}$ as in
the theorem statement.

\subsection{Polynomial lattices}
\label{subsec:polylat}

A lattice is a \emph{module} over the ring $\ZZ$ of integers.  In other
words, not only is it an abelian group under addition, but we can also
multiply lattice vectors by integers and thus take arbitrary integer
combinations of them.  More generally, a module for a ring $R$ is an
abelian group in which we can multiply by elements of $R$ (in a way
that satisfies the associative and distributive laws).  In other words,
an $R$-module is exactly like an $R$-vector space, except that $R$ is
not required to be a field, as it is in the definition of a vector
space.

The module $R^m$ with componentwise scalar multiplication is called a
\emph{free} $R$-module of rank $m$.  Every lattice is a free
$\ZZ$-module, and free $R$-modules will be the analogous structure for
the ring $R$.

For example, if $R$ is the polynomial ring $F[z]$ over a field $F$,
then we define a \emph{polynomial lattice} to be a free module over
$F[z]$ of finite rank.  A polynomial lattice will usually be generated
by a basis of vectors whose coefficients are polynomials in $z$.
Vectors in our polynomial lattice will be linear combinations of the
basis vectors (where the coefficients are also polynomials in $z$).

An appropriate definition of the length (i.e., degree) of such a lattice
vector is the maximum degree of its coordinates:
\begin{equation} \label{eq:normdef}
\deg_z (v_1(z), v_2(z), \ldots, v_m(z)) = \max_i \deg_z v_i(z).
\end{equation}
This defines a non-Archimedean norm.  In fact, for lattices with a norm
defined as above, it is possible to find the exact shortest vector in
polynomial time (see, for example, \cite{vonzurgathen}).

Lattices of polynomials have been well studied because of their
applications to the study of linear systems \cite{Kailath}.  There are
several notions of basis reduction for such lattices.  A basis is {\em
column-reduced} (or, as appropriate, {\em row-reduced}) if the degree
of the determinant of the lattice (i.e., of a basis matrix) is equal to
the sum of the degrees of its basis vectors.  Such bases always contain
a minimal vector for the lattice, and $m$-dimensional column reduction
can be carried out in $m^{\omega+o(1)} D$ field operations \cite{GJV},
where $\omega$ is the exponent of matrix multiplication and $D$ is the
greatest degree occurring in the original basis of the lattice.

In particular, for an $m$-dimensional lattice $L$ with the
norm~\eqref{eq:normdef}, the above algorithms are guaranteed to find a
nonzero vector $v$ for which
\begin{equation} \label{eqn:degdetL}
\deg v \le \frac{1}{m} \deg \det L,
\end{equation}
where $\det L$ denotes the determinant of a lattice basis.

\subsection{Finding short vectors under general non-Archimedean norms}
\label{subsec:findingshort}

The above algorithms are specialized to norms defined by \eqref{eq:normdef},
but there are other non-Archimedean norms, and we will need to use them in
the proof of Theorem~\ref{theorem:ff-coppersmith} in the function field
setting. In fact, we will show that for all non-Archimedean norms, one can
find a vector satisfying the equivalent of \eqref{eqn:degdetL} in a lattice
by solving a system of linear equations.  Solving this system may be less
efficient than a specialized algorithm, but it allows us to give a general
approach that will work in polynomial time for any norm.

Let $R = F[z]$ be a polynomial ring over a field $F$, and for $r \in R$
define
\[
|r| = C^{\deg_z(r)}
\]
for some arbitrary constant $C>1$; we take $|0|=0$ as a special case. Note
that $|z| = C$, and thus we can write $|r| = |z|^{\deg_z(r)}$.

Suppose we have any norm $|\cdot|$ on $R^m$ that satisfies the
following three properties:
\begin{enumerate}
\item For all $v \in R^m$, $|v| \ge 0$, and $|v|=0$ if and only if
    $v=0$.

\item For all $v,w \in R^m$, $|v+w| \le \max(|v|,|w|)$.

\item For all $v \in R^m$ and $r \in R$, $|rv| = |r||v|$.
\end{enumerate}
Note that taking
\[
|(v_1(z), v_2(z), \ldots, v_m(z)| = C^{\max_i \deg_z v_i(z)}
\]
defines such a norm, but the extra generality will prove useful in
Section~\ref{section:functionfield}.

Let $M \subseteq R^m$ be a submodule of rank $m$ (so the quotient
$F$-vector space $R^m/M$ is finite-dimensional), and let $d = \dim_F
(R^m/M)$.

\begin{lemma} \label{lemma:small}
For any $R$-basis $b_1,\dots,b_m$ of $R^m$, there exists a nonzero
vector $v \in M$ such that
\[
|v| \le \sqrt[m]{|b_1|\dots|b_m|} \, |z|^{d/m}.
\]
\end{lemma}

\begin{proof}
We will construct a nonzero vector satisfying $|v| \le q^c$ for some
constant $c$ to be determined, and then we will optimize the choice of
$c$.  Let $|b_i| = |z|^{n_i}$, and consider the space of polynomials
\[
V = \left\{\sum_i r_i b_i :
\text{$r_i \in R$ and $\deg_z r_i \le {c-n_i}$} \right\}.
\]
Every $v \in V$ satisfies $|v| \le |z|^c$, and $V$ is an $F$-vector
space. To compute its dimension, note that $r_i$ is determined by
$\lfloor c-n_i \rfloor + 1 > c-n_i$ coefficients. Because
$b_1,\dots,b_m$ is an $R$-basis, $\dim_F V > mc - \sum_i n_i$.

If we take $c = \big(d+\sum_i n_i\big)/m$, then $\dim_F V > d$.  Thus,
there exists a nonzero element $v$ of $V$ that maps to zero in the
$d$-dimensional quotient space $R^m/M$ and hence lies in $M$. It
satisfies
\[
|v| \le q^c = \sqrt[m]{|b_1|\dots|b_m|} \, |z|^{d/m},
\]
as desired.
\end{proof}

\begin{lemma} \label{lemma:findallsmall}
Under the hypothesis of Lemma~\ref{lemma:small}, a vector satisfying
\[
|v|\le \sqrt[m]{|b_1|\dots|b_m|} \, |z|^{d/m}
\]
can be found in polynomial time (given an $R$-basis of $M$).
\end{lemma}

\begin{proof}
In the notation of the proof of Lemma~\ref{lemma:small}, we will show
that we can find small coefficients $r_1,\dots,r_m \in R$ (not all
zero) such that $\sum_i r_i b_i$ is in $M$.  Suppose $w_1,\dots,w_m$ is
an $R$-basis of $M$. Then the elements of $M$ are those that can be
written as $\sum s_i w_i$ with $s_i \in R$. Given a polynomial bound
for the degrees of $s_1,\dots,s_m$, we could determine the coefficients
$r_i$ and $s_i$ by solving linear equations over $F$ for their
coefficients. To specify these equations, we write $w_1,\dots,w_m$ as
$R$-linear combinations of $b_1,\dots,b_m$.  Define the matrix $W$ over
$R$ by $w_j = \sum_i W_{ij} b_i$ for each $j$.  Then
\[
\sum_i r_i b_i = \sum_j s_j w_j
\]
amounts to $r = W s$, where $s$ and $r$ are the column vectors with
entries $s_i$ and $r_i$, respectively.

Thus, $s$ determines $r$ in a simple way, and all we need is to choose
$s_1,\dots,s_m$ so that setting $r = Ws$ yields $\deg_z r_i \le c-n_i$, with
$c$ and $n_i$ defined as in the proof of Lemma~\ref{lemma:small}.  It is not
difficult to bound the degrees of the polynomials $s_i$ as follows. Let
$\widetilde{W}$ be the adjoint matrix of $W$ (so $W \widetilde{W} = \det(W)
I$).  Then
\[
\widetilde{W} r = \det(W) s.
\]
It follows that for each $i$,
\[
\deg_z \det(W) + \deg_z s_i  \le
\max_{j} \big(\deg_z\widetilde{W}_{ij} + \deg_z r_j\big).
\]
However, the entries $\widetilde{W}_{ij}$ of $\widetilde{W}$ have
degree bounded by $m-1$ times the maximum degree of an entry of $W$
(because they are given by determinants of $(m-1) \times (m-1)$
submatrices of $W$).  Thus, $\deg_z s_i $ is polynomially bounded, and
we can locate a suitable vector $v$ by solving a system of polynomially
many linear equations over $F$.
\end{proof}

Note that for a rank $m$ submodule $M$ of $R^m$, the degree of the
determinant of a basis matrix $B$ for $M$ is the dimension of the
quotient $R^m/M$.  Thus, in Lemma~\ref{lemma:small}, if
$|b_1|=\dots=|b_m|=1$, then the norm of a minimal vector is bounded by
$|\det(B)|^{1/m}$.  The exponential approximation factor that occurs in
LLL lattice basis reduction does not occur here.

\section{Coppersmith's theorem}
\label{section:coppersmith}

We now review how Coppersmith's method works over the integers, as this
provides a template for the techniques we will apply later.  We will
follow the exposition of May \cite{may:survey}.

Let $f(x)$ be a monic univariate polynomial of degree $d$, and $N$ an
integer of potentially unknown factorization.  We wish to find all
small integers $w$ such that $\gcd(f(w),N)$ is large.

To do so, we will choose some positive integer $k$ (to be determined
later) and look at integer combinations of the polynomials $x^j f(x)^i
N^{k-i}$. If $B$ divides both $N$ and $f(w)$, then $B^k$ will divide
$w^j f(w)^i N^{k-i}$ and thus also any linear combination of such
polynomials.

Let
\[
Q(x) = \sum_{i,j} a_{i,j} x^j f(x)^i N^{k-i} = \sum_i q_i x^i,
\]
for some coefficients $a_{i,j}$ and $q_i$ to be determined. We will
choose $Q$ so that the small solutions to our original congruence
become actual solutions of $Q(x) = 0$ in the integers. This will allow
us to find $w$ by factoring $Q(x)$ over the rationals. The construction
of $Q$ tells us that
\begin{equation}
\label{eq:integers-vanish}
Q(w) \equiv 0 \pmod{B^k}.
\end{equation}
If in addition we have a lower bound $N^\beta$ on the size of $B$, and
we can show that
\begin{equation}
\label{eq:integers-Q<B}
|Q(w)| < N^{\beta k} \le  B^k,
\end{equation}
then $Q(w) = 0$ and we may find $w$ by factoring $Q$.  In fact, this
observation tells us that we can find \emph{all} such $w$ in this way.
A similar observation will appear in all of our proofs.

In the case of the integers, we introduce the bound $|w| < X$ on our
roots, and the triangle inequality tells us that
\begin{equation}
\label{eq:integers-q_i}
|Q(w)| \leq \sum_i |q_i| X^i.
\end{equation}
To finish the theorem, we will show that if $X$ is sufficiently small,
then we can choose $Q$ so that its coefficients $q_i$ satisfy
\begin{equation}
\label{eq:integers-q<Nb}
\sum_i |q_i|X^i < N^{\beta k}.
\end{equation}

We are now ready to prove Coppersmith's theorem for the integers.

\begin{proof}[Proof of Theorem \ref{theorem:coppersmith}]
Having outlined the general technique above, it remains to be shown
that we can construct a polynomial $Q(x)$ whose coefficients satisfy
the bound in \eqref{eq:integers-q<Nb}.

The polynomial $Q(x)$ will be a linear combination of the polynomials
\[
x^j f(x)^i N^{k-i} \quad \textup{for} \quad 0 \leq i < k \textup{ and } 0 \leq j
< d
\]
and
\[
x^j f(x)^k \quad \textup{for} \quad 0 \leq j < t.
\]

The right-hand side of \eqref{eq:integers-q_i} is the $\ell_1$ norm of
the vector of coefficients of the polynomial $Q(xX)$, which in turn
will be a linear combination of the polynomials $(xX)^j f(xX)^i
N^{k-i}$. Finding our desired $Q(x)$ is thus equivalent to finding a
suitably short vector in the lattice $L$ spanned by the coefficient
vectors of the polynomials $(xX)^j f(xX)^i N^{k-i}$.
Once we find this short vector, we can divide each coefficient by the power of $X$
introduced in the normalization to find the coefficients of $Q$, and test each of
the roots of $Q$ to see if it is a solution.

To compute the determinant of this lattice, we can order the basis
vectors by the degrees of the polynomials they represent to obtain an
upper triangular matrix whose determinant is the product of the terms
on the diagonal:
\[
\det(L) = \prod_{0 \leq i < dk+t} X^i \prod_{0 \leq j \leq k}
N^{dj} = X^{(dk+t-1)(dk+t)/2} N^{dk(k+1)/2}.
\]

Set $m=dk+t$.  We can use the LLL algorithm \cite{lll} to find a vector
$v$ whose $\ell_2$ norm is bounded by
\[
|v|_2 \leq 2^{(m-1)/4} \det(L)^{1/m}.
\]
By the Cauchy-Schwarz inequality, $|v|_1 \leq \sqrt{m}\, |v|_2$, and hence
whenever $|w| < X$,
\[
|Q(w)| \leq \sqrt{m} 2^{(m-1)/4} \det(L)^{1/m}.
\]
We assume $m \ge 7$, and use the weaker bound
\[
|Q(w)| \leq 2^{(m-1)/2} \det(L)^{1/m}.
\]

To prove inequality \eqref{eq:integers-Q<B}, we must show that
\[
2^{(m-1)/2} \left( X^{m(m-1)/2} N^{dk(k+1)/2}
\right)^{1/m} < N^{\beta k}.
\]
This inequality is equivalent to
\begin{equation}
\label{eq:integer-condition}
(2X)^{(m-1)/(2k)} N^{d(k+1)/(2m)}
 < N^{\beta}.
\end{equation}

Applying Lemma~\ref{lem:k-and-m} below with $\ell = \log_2 2X$ and $n =
\log_2 N$, we obtain parameters $k$ and $t$ such that
\eqref{eq:integer-condition} holds for
\[
2X < N^{{\beta^2}/{d} - \varepsilon}.
\]
To eliminate $\varepsilon$ from the statement of the theorem, take
$\varepsilon < \frac{1}{\log_2 N}$.  Then it suffices to take $X \le
\frac{1}{4} N^{\beta^2/d}$.  We can divide the interval $[-4X,4X]$ into four
intervals of width $2X$ and solve the problem for each interval by finding
solutions for the polynomials $f(x-3X)$, $f(x-X)$, $f(x+X)$, and $f(x+3X)$.
Thus, we achieve a bound of $N^{\beta^2/d}$, as desired.
\end{proof}

We end with a brief lemma that will tell us how to optimize our
parameters in equation \eqref{eq:integer-condition}.

\begin{lemma}
\label{lem:k-and-m} The inequality $\ell \frac{m-1}{2k} + nd\frac{k+1}{2m} <
n \beta$ is satisfied when $\ell < n \left( \frac{\beta^2}{d} - \varepsilon
\right)$, $m \ge
\max\left(\frac{2\beta}{\varepsilon},\frac{2d}{\beta}\right)$, and $k =
\left\lfloor \frac{\beta m}{d} - 1 \right\rfloor$.
\end{lemma}

Note that for the application above, we must have $k \ge 1$ and $t = m-dk \ge
0$.  The hypotheses of the lemma achieve this.  Furthermore, we want $m$ and
$k$ to be polynomially bounded.  Without loss of generality, we can assume
that $N^{\beta^2/d} \ge 2$, and hence $\beta^2 \ge d/n$. Thus, as long as
$\varepsilon$ is not too small, $m$ and $k$ need not be too large.

Lemma~\ref{lem:k-and-m} amounts to optimizing how large $\ell$ can be. As
intuition, note that if we set the two terms $\ell \frac{m-1}{2k}$ and
$nd\frac{k+1}{2m}$ roughly equal to $\frac{n\beta}{2}$, then we have $\ell
m^2 \approx ndk^2 \approx n\beta m k$ and hence $\ell \approx n \beta^2/d$.
The proof amounts to making this precise.

\begin{proof}
It suffices to show that these values of $m$ and $k$ satisfy $
n\left(\frac{\beta^2}{d}-\varepsilon\right)\frac{m-1}{2k} <
\frac{n\beta}{2} $ and $ nd\frac{k+1}{2m} \le \frac{n\beta}{2} $.

The first inequality is equivalent to $\frac{k}{m-1} > \frac{\beta}{d}
- \frac{\varepsilon}{\beta}$. Similarly, the second is equivalent to
$\frac{k+1}{m} \le \frac{\beta}{d}$.  If we set $k = \left\lfloor
\frac{\beta m}{d} - 1 \right\rfloor$, then $\frac{k+1}{m} \le
\frac{\beta}{d}$, so the second inequality is satisfied.  If in
addition we take $m \ge \frac{2\beta}{\varepsilon}$, then
$\frac{\varepsilon m}{\beta} \ge 2$ and hence $k > \frac{\beta m}{d} -
2 \ge \frac{\beta m}{d} - \frac{\varepsilon m}{\beta}$.  It follows
that $k \frac{m}{m-1} > \frac{\beta m}{d} - \frac{\varepsilon
m}{\beta}$, which is equivalent to the first inequality.
\end{proof}

Note that improving the approximation factor for the length of the short
lattice vector that we find will only improve the constants and running time
of the theorem, and will not provide an asymptotic improvement to the bound
$N^{\beta^2/d}$ on $|w|$.

\section{Polynomials and Reed-Solomon list decoding}
\label{section:polycop}

In this section, we prove Theorem~\ref{theorem:polycop} using an
approach analogous to that of the previous section.  Guruswami and
Sudan's technique for list decoding of Reed-Solomon codes
\cite{guruswami-sudan} is similar in that it involves constructing a
bivariate polynomial that vanishes to high order at particular points.
To construct such a polynomial, they write each vanishing condition as
a set of linear equations on the coefficients of the polynomial under
construction. The linear equations can be solved to obtain the desired
polynomial, and the polynomial factored to obtain its roots.

Similarly, the polynomials used in Coppersmith's method are constructed so as
to vanish to high order, the condition ensured by equation
\eqref{eq:integers-vanish}.  The conceptual difference is that this condition
follows from the form of the lattice basis, rather than being imposed as
linear constraints.  With the right definition of lattice basis reduction in
the polynomial setting, we can emulate the proof from the integer case.

We regard $f(x)$ as a polynomial in $x$ with coefficients that are
polynomials in the variable $z$.  To prove
Theorem~\ref{theorem:polycop}, we would like to construct a polynomial
$Q(x)$ over $F[z]$ from the polynomials $x^j f(x)^i p(z)^{k-i}$.  If
$b(z)$ divides both $p(z)$ and $f(w(z))$, then $b(z)^k$ divides $w(z)^j
f(w(z))^i p(z)^{k-i}$ and thus also any linear combination of such
polynomials.

Instead of an integer combination of these polynomials, we will allow
coefficients that are polynomials in $z$. Let
\[
Q(x) = \sum_{i,j} a_{i,j}(z) x^j f(x)^i p(z)^{k-i} = \sum_i q_i(z) x^i.
\]

If we have an upper bound $\ell$ on the degree of our root $w(z)$, then the
degree of $Q(w(z))$ will be bounded by
\[
\deg_z Q(w(z)) \le \max_i \,\,(\deg_z q_i(z) + \ell i).
\]
If similarly we have a lower bound $n\beta$ on the degree of $b(z)$,
then if we know that both
\[
Q(w(z)) \equiv 0 \pmod{b(z)^k}
\]
and
\begin{equation}
\label{eq:polynomials-Q<B}
\deg_z Q(w(z)) < n \beta k \le k \deg_z b(z),
\end{equation}
then we may conclude that
\[
Q(w(z)) = 0.
\]

\subsection{Proof of Theorem~\ref{theorem:polycop}}

We will show how finding a short vector in a lattice of polynomials
will allow us to construct a polynomial $Q(x)$ satisfying
\eqref{eq:polynomials-Q<B}.

Let $\ell$ be the upper bound on the degree of the roots $w(z)$ we
would like to find.  Using the same idea to bound the length of the
vector as in the integer case, we will form a lattice of the
coefficient vectors of
\[
(z^\ell x)^j f(z^\ell x)^i p(z)^{k-i} \quad
\textup{for} \quad 0 \leq j < d \textup{ and } 0 \leq i < k
\]
and
\[
(z^\ell x)^j f(z^\ell x)^k \quad \textup{for} \quad 0 \leq j < t.
\]
As always, we view them as polynomials in powers of $x$ with
coefficients that are polynomials in $z$.
Once we find this short vector, we can divide each coefficient by the power of $z^\ell$
introduced in the normalization to find the coefficients of $Q$, and test each of
the roots of $Q$ to see if it is a solution.

Let $M$ be the $F[z]$-module
spanned by the coefficient vectors of these polynomials, with the
degree of a vector defined by \eqref{eq:normdef}.

The matrix of coefficient vectors of the basis is upper triangular, so its
determinant is the product of the diagonal entries.  Set $m = kd+t$.  Then
\begin{align*}
\deg \det M & = \ell \sum_{i=0}^{m-1} i + nd \sum_{i=0}^k i \\
&= \ell \frac{m(m-1)}{2} + nd\frac{k(k+1)}{2}.
\end{align*}

Since the dimension of our lattice is $m$, by
Theorem~\ref{lemma:findallsmall} we can find a vector of degree at most
\[
\frac{1}{m} \left( \ell \frac{m(m-1)}{2} + nd\frac{k(k+1)}{2} \right).
\]

To prove \eqref{eq:polynomials-Q<B}, we would like this bound to be less than
$\beta k n$. By Lemma \ref{lem:k-and-m}, we can achieve any $\ell \leq n
\left( \frac{\beta^2}{d} - \varepsilon \right)$. If we take $\varepsilon <
\frac{1}{n^2 d}$ then this becomes $\ell < \frac{\beta^2 n}{d}$, as desired,
because $\beta$ can be taken to have denominator $n$.

Note that we cannot achieve degree equal to $\beta^2n/d$ (as opposed to
strict inequality): for the equation $x^d \equiv 0 \pmod{p(z)^d}$, there are
infinitely many solutions $x = c \,p(z)$ if the field $F$ is infinite, so it
is impossible to list them all in polynomial time.

\subsection{Reed-Solomon list decoding and noisy polynomial interpolation}
\label{section:rsld}

A Reed-Solomon code is determined by evaluating a polynomial $w(z) \in
\FF_q[z]$ of degree at most $\ell$ at a collection of distinct points $(x_1,
\ldots, x_n)$ to obtain a codeword $(w(x_1), \ldots, w(x_n))$. In the
Reed-Solomon decoding problem, we are provided with $(y_1, \ldots, y_n)$,
where at most $e$ entries in the codeword have changed, and we wish to
recover $w(z)$ by finding a polynomial of degree at most $\ell$ that fits at
least $n-e$ points $(x_i,y_i)$. Guruswami and Sudan \cite{guruswami-sudan}
showed how to correct up to $e=n-\sqrt{n\ell}$ errors by providing a list of
all possible decodings.

In the noisy polynomial interpolation problem, at
each $x_i$ a set $\{y_{i1}, \ldots, y_{id} \}$ of values is
specified, and the goal is to find a low-degree polynomial passing
through a point from each set.  This problem has been proposed as a
cryptographic primitive, for example by Naor and Pinkas
\cite{naor-pinkas}, and studied by Bleichenbacher and Nguyen
\cite{bleichenbacher-nguyen}.

We can use Theorem \ref{theorem:polycop} to solve both problems, and in
particular recover the exact decoding rates of Guruswami and Sudan. The input
to our problem is a collection of points
\[
\{ (x_i, y_{ij}) : 1
\le i \le n, 1 \le j \le d \}.
\]
We set $p(z) = \prod_i (z-x_i)$, and we define a monic polynomial
$f(x)$ of degree $d$ in $x$ by
\[
f(x) = \sum_{i=1}^n \prod_{j=1}^d (x-y_{ij})
\prod_{\stackrel{\scriptstyle k=1}{\scriptstyle k \ne i}}^n
\frac{z-x_k}{x_i-x_k}.
\]
We have constructed $f(x)$ by interpolation so that $f(x) \equiv
\prod_j (x - y_{ij}) \pmod{(z - x_i)}$.  Thus, $f(y_{ij})=0$ whenever
$z = x_i$.

To correct $e$ errors, we seek a polynomial $w(z)$ of degree at most
$\ell$ such that for at least $n-e$ values of $i$, there exists a $j$
such that $w(x_i) = y_{ij}$.  In other words, $f(w(z))$ must be
divisible by at least $n-e$ factors $z-x_i$, which is equivalent to
\[
\deg_z \gcd(f(w(z)),p(z)) \ge n-e.
\]
By Theorem~\ref{theorem:polycop}, we can solve this problem in polynomial
time if $\ell < n(1-e/n)^2/d$ (since $\beta = 1-e/n$ in the notation of the
theorem).  That is equivalent to the Guruswami-Sudan bound $e < n - \sqrt{n
\ell d}$.

\subsection{Running time}

The Guruswami-Sudan algorithm consists of two parts: constructing the
polynomial $Q(x)$, and finding the roots of $Q(x)$ in $\FF_q[z]$.  In
this paper, we do not address the second part, but we improve the
running time of the first part, which has been the bottleneck in the
algorithm.

The time to construct $Q$ is dominated by the lattice basis reduction step,
which depends on the dimension $m$ of the lattice and the maximum degree $D$
of a coefficient polynomial.  In our construction, we have $D = O(nk)$.

Using the fastest row reduction algorithm (see
Section~\ref{subsec:polylat}), the running time is
\[
O\big(D m^{\omega+o(1)}\big)= O\big(nk m^{\omega + o(1)}\big).
\]
With cubic-time matrix multiplication we achieve $O(nkm^3)$, and with fast matrix multiplication~\cite{VW} we achieve $O(nkm^{2.3727})$.

The fastest previous algorithm proposed for this problem from Beelen and Brander~\cite{beelen-brander} runs in time $O(m^4 k n \log^2 n \log \log n)$.

\section{Number fields}
\label{section:numberfield}

\subsection{Background on number fields}

See \cite{lenstra:algorithms} for a beautiful introduction to
computational algebraic number theory, or \cite{cohen} for a more
comprehensive treatment.

Recall that number fields are finite extensions of the field $\QQ$ of
rational numbers.  Each number field $K$ is generated by some algebraic
number $\alpha$, and the elements of the number field are polynomials
in $\alpha$ with rational coefficients. If the minimal polynomial
$p(x)$ of $\alpha$ (the lowest-degree polynomial over $\QQ$, not
identically zero, for which $\alpha$ is a root) has degree $n$, then
every element of $K=\QQ(\alpha)$ will be a polynomial in $\alpha$ of
degree at most $n-1$.  In other words,
$$
\QQ(\alpha) = \{a_0 + a_1\alpha + \dots +
a_{n-1}\alpha^{n-1} : a_0,\dots,a_{n-1} \in \QQ\}.
$$
The degree of $K$ is defined to be $n$.  It is the dimension of $K$ as
a $\QQ$-vector space.

The minimal polynomial $p(x)$ must be irreducible over $\QQ$, and thus
it has $n$ distinct complex roots $\alpha_1,\dots,\alpha_n$ (one of
which is $\alpha$).  Not all of these roots will necessarily be in the
field $K = \QQ(\alpha)$.  For example, the field $\QQ(\sqrt[3]{2})$ is
contained in $\RR$ and thus does not contain either of the complex
roots of $x^3-2$.

For each $i$ from $1$ to $n$, we can define an embedding $\sigma_i$ of
$K$ into $\CC$ by mapping $\alpha$ to $\alpha_i$ and extending by
additivity and multiplicativity.  All embeddings into $\CC$ arise in
this way. If $p$ has $r_1$ real roots and $r_2$ pairs of complex
conjugate (non-real) roots, then there will be $r_1$ real embeddings
and $2r_2$ complex embeddings.

The Archimedean absolute values on $K$ are defined by
$$
|\gamma|_i = |\sigma_i(\gamma)|
$$
(where $|\cdot|$ on the right side is the familiar absolute value on
$\CC$, and $|\cdot|_i$ does not denote the $\ell_i$ norm). For each
$i$, this valuation has all the usual properties of the absolute value
on $\QQ$. These absolute values are not necessarily distinct, since
they coincide for complex conjugate roots of $p(x)$: if $\alpha_i =
\overline{\alpha_j}$, then $|\gamma|_i = |\gamma|_j$ for all $\gamma$.
Otherwise, the absolute values are all distinct.

The ring of algebraic integers $\OO_K$ in $K$ consists of all the
elements of $K$ that are roots of monic polynomials over $\ZZ$.  It is
the natural analogue of $\ZZ$ in $K$ (note that $\OO_\QQ = \ZZ$).  In
simple cases, $\OO_K$ may equal $\ZZ[\alpha]$, but that is not always
true. When $K = \QQ(\sqrt{5})$, we have $\OO_K = \ZZ[(1+\sqrt{5})/2]$,
and for some number fields the ring of integers cannot even be
generated by a single element.

The norm of an element $\gamma \in K$ is defined as the product
\[
N(\gamma) = \sigma_1(\gamma) \dots \sigma_n(\gamma)
\]
in $\CC$.  (In fact, $N(\gamma)$ is rational for $\gamma \in K$, and it
is integral for $\gamma \in \OO_K$.)  If $\gamma \in \OO_K$ and $\gamma
\ne 0$, then $|N(\gamma)| = |\OO_K/\gamma \OO_K|$. More generally, for
any nonzero ideal $I$ in $\OO_K$, we define its norm $N(I)$ to be
$|\OO_K/I|$.  The norm is multiplicative; i.e., $N(IJ) = N(I) N(J)$.

The norm is a natural measure of size for both ideals and individual
elements in $\OO_K$.  It might be tempting to use the norm as our
measure of the size of the roots of the polynomial in
Theorem~\ref{theorem:nf-coppersmith}.  However, that does not work,
because $\OO_K$ typically has infinitely many units (elements of
norm~$1$).  For example, the powers of $(1+\sqrt{5})/2$ are units in
$\ZZ[(1+\sqrt{5})/2]$, which means the equation $x^2 \equiv 0 \pmod{4}$
has infinitely many solutions of norm at most $N(4)^{1/2}=N(2)=4$,
namely the numbers $2((1+\sqrt{5})/2)^k$ for $k \in \ZZ$.  Thus,
bounding the norm alone is insufficient even to guarantee that there
will be only finitely many solutions, but bounding all the absolute
values suffices.

The ring $\OO_K$ has an integral basis $\omega_1, \ldots, \omega_n$
(i.e., a basis such that every element of $\OO_K$ can be expressed
uniquely in the form $\sum_i a_i \omega_i$ with $a_i \in \ZZ$). We
assume we are given such a basis, because finding one is
computationally difficult (see Theorem~4.4 in
\cite{lenstra:algorithms}).  Any reasonably explicit description of
$\OO_K$ will yield an integral basis. Fortunately, such a description
is known for many concrete examples of number fields, such as
cyclotomic fields.  Furthermore, if we are working with a fixed number
field, finding an integral basis for $\OO_K$ can be done with only a
fixed amount of preprocessing.  We also assume that ideals in $\OO_K$
are given in terms of integral bases. It is not difficult to convert
any other description of an ideal (such as generators over $\OO_K$) to
an integral basis.

If we do not know the full ring $\OO_K$ of integers, we could
nevertheless work with an order in $K$, i.e., a finite-index subring of
$\OO_K$. Everything we need works just as well for orders, with one
exception, namely that the norm is no longer multiplicative for ideals.
Fortunately, it remains multiplicative for invertible ideals (see
Proposition~4.6.8 in \cite{cohen}), and Coppersmith's theorem
generalizes to invertible ideals.  Specifically, we can find small
roots of polynomial equations modulo an invertible ideal $I$, or modulo
any invertible ideal $B$ that contains $I$ and satisfies $N(B) \ge
N(I)^\beta$.

Polynomials over number fields can be factored in polynomial time
\cite{lenstra:factoring}.

\subsubsection{Modules and canonical embeddings}

The analogue of a lattice for $\OO_K$ is a finitely generated
$\OO_K$-submodule of the $r$-dimensional $K$-vector space $K^r$. Recall
that an $\OO_K$-submodule is a non-empty subset that is closed under
addition and under multiplication by any element in $\OO_K$.

Unlike the case of $\ZZ$-lattices, $\OO_K$-lattices may not have bases
over $\OO_K$.  However, an $\OO_K$-lattice $\Lambda$ always has a
pseudo-basis, i.e., a collection of vectors $v_1,\dots,v_s \in \Lambda$
and ideals $I_1,\dots,I_s \subseteq \OO_K$ such that
\[
\Lambda = I_1 v_1 + \dots + I_s v_s.
\]
The key difference from $\ZZ$ is that the ideals may not be principal
(i.e., they may not simply be the multiples of single elements of
$\OO_K$).

A natural approach to finding a short vector in an $\OO_K$-lattice would be
to find an algorithm to reduce a pseudo-basis.  Fieker and Pohst
\cite{fieker-pohst} developed an $\OO_K$-analogue of the LLL lattice basis
reduction algorithm, but they were unable to prove that their algorithm runs
in polynomial time.  More recently, Fieker and Stehl\'e \cite{fieker-stehle}
have given a polynomial-time algorithm to find a reduced pseudo-basis in an
$\OO_K$-module.  Their algorithm runs in two parts.  The first applies LLL to
an embedding of the $\OO_K$-module as a $\ZZ$-lattice to find a full-rank set
of short module elements, and the second uses this collection of module
elements to reduce the pseudo-basis.

As our application only requires finding a short vector in the module,
we do not need the second step of the Fieker-Stehl\'e algorithm.  The
remainder of this section describes how to use LLL to find a short
vector in an $\OO_K$-lattice.

Although $\OO_K$-lattices are an algebraic analogue of $\ZZ$-lattices,
their geometry is not as easy to see directly from the definition.  It
might seem natural simply to use one of the absolute values to define
the $\ell_2$ norm for vectors, but that breaks the symmetry between
them.  Instead, it is important to treat each absolute value on an
equal footing, and the canonical embedding (defined below) allows us to
do so.

We will describe the embedding in several steps.  First, we embed
$\OO_K$ itself as an $n$-dimensional lattice in $\RR^{r_1} \oplus
\CC^{2r_2}$ by mapping $\gamma \in \OO_K$ to $
(\sigma_1(\gamma),\dots,\sigma_n(\gamma)). $ An integral basis
$\omega_1,\dots,\omega_n$ of $\OO_K$ is mapped to the rows of the
matrix
\[
\sigma(\omega)=
\begin{pmatrix}
\sigma_1(\omega_1) & \sigma_2(\omega_1) & \cdots & \sigma_n(\omega_1) \\
\sigma_1(\omega_2) & \ddots & \ & \sigma_n(\omega_2)\\
\vdots && \ddots &  \vdots \\
\sigma_1(\omega_n) & \sigma_2(\omega_n) & \cdots & \sigma_n(\omega_n)
\end{pmatrix},
\]
so $\OO_K$ is mapped to the $\ZZ$-linear combinations of the rows.

The discriminant $\Delta_K$ of $K$ is defined by
\[
\Delta_K = \det \sigma(\omega)^2.
\]
It is an integer that measures the size of the ring of integers in $K$.

The canonical embedding of the principal ideal generated by an element
$\gamma$ is generated by the rows of the matrix product
\[
\begin{pmatrix}
\sigma_1(\omega_1) & \sigma_2(\omega_1) & \cdots & \sigma_n(\omega_1) \\
\sigma_1(\omega_2) & \ddots & \ & \sigma_n(\omega_2)\\
\vdots && \ddots &  \vdots \\
\sigma_1(\omega_n) & \sigma_2(\omega_n) & \cdots & \sigma_n(\omega_n)
\end{pmatrix}
\begin{pmatrix}
\sigma_1(\gamma) \\
 & \sigma_2(\gamma) \\
 && \ddots \\
 &&& \sigma_n(\gamma)
\end{pmatrix}.
\]

More generally, suppose we have an ideal $B$ generated by an integral
basis $b_1, \ldots, b_n$. Let $M_B$ be the matrix defined by
\[
b_i = \sum_j \big(M_B\big)_{ij} \omega_j.
\]
The canonical embedding of $B$ is generated by the rows of
\[
\sigma(b) =
\begin{pmatrix}
\sigma_1(b_1) & \sigma_2(b_1) & \cdots & \sigma_n(b_1) \\
\sigma_1(b_2) & \ddots & \ & \sigma_n(b_2)\\
\vdots && \ddots&  \vdots \\
\sigma_1(b_n) & \sigma_2(b_n) & \cdots & \sigma_n(b_n)
\end{pmatrix}
=
M_B
\begin{pmatrix}
\sigma_1(\omega_1) & \sigma_2(\omega_1) & \cdots & \sigma_n(\omega_1) \\
\sigma_1(\omega_2) & \ddots & \ & \sigma_n(\omega_2)\\
\vdots && \ddots&  \vdots \\
\sigma_1(\omega_n) & \sigma_2(\omega_n) & \cdots & \sigma_n(\omega_n)
\end{pmatrix}.
\]
Note that the absolute value of the determinant of $\sigma(b)$ equals
$|\det M_B| \sqrt{|\Delta_K|}$, and $|\det M_B| = |\OO_K/B| = N(B)$.

Finally, we can easily extend the canonical embedding from $\OO_K$ to
$\OO_K^{\,r}$ by embedding each of the $r$ coordinates independently.
Given a pseudo-basis $v_1,\dots,v_r$ with corresponding ideals
$I_1,\dots,I_r$, the canonical embedding of the lattice is generated by
the rows of the block matrix whose $ij$ block of size $n \times n$ is
equal to
\[
M_{I_i} \sigma(\omega) \begin{pmatrix}
\sigma_1(v_{ij}) \\
 & \sigma_2(v_{ij}) \\
 && \ddots \\
 &&& \sigma_n(v_{ij})
\end{pmatrix},
\]
where $v_{ij}$ is the $j$-th component of $v_i$.

The inner product on $\RR^{r_1} \oplus \CC^{2r_2}$ is given by the
usual dot product on $\RR$ and the Hermitian inner product on $\CC$
(i.e., $\langle x,y \rangle = x \overline{y}$ for $x,y \in \CC$). Thus,
it is positive definite.

The canonical embedding's image lies within an $n$-dimensional real
subspace, because the complex embeddings come in conjugate pairs.  In
fact, we can transform it into a simple real embedding. To do so,
consider the $r_2$ pairs of complex embeddings.  For each pair
$(\sigma_j(\gamma), \sigma_k(\gamma))$ of complex embeddings that are
conjugates of each other, we can map the pair $(\sigma_j(\gamma),
\sigma_k(\gamma))$ to $(\sqrt{2} \Re(\sigma_j(\gamma)),
\sqrt{2}\Im(\sigma_j(\gamma)))$. The reason for the factor of
$\sqrt{2}$ is to ensure that the inner product is preserved.
Furthermore, the absolute value of the determinant is preserved.

Once we have a real embedding of our $\OO_K$-lattice, we can apply the
LLL algorithm to find a short vector in the real embedded lattice,
which will correspond to a short vector in the original
$\OO_K$-lattice.  Unfortunately, using LLL in the canonical embedding
does not preserve the $\OO_K$-structure, so it does not produce a
reduced pseudo-basis over $\OO_K$, but a short vector is sufficient for
our purposes here.

\subsection{Proof of Theorem~\ref{theorem:nf-coppersmith}}

The following lemma is the analogue of the statement over the integers
that a multiple of $n$ that is strictly less than $n$ in absolute value
must be zero.

\begin{lemma}
\label{lem:ideal} For a nonzero ideal $I$ in $\OO_K$ and an element
$\gamma \in I$, if $|N(\gamma)| < N(I)$ then $\gamma = 0$.
\end{lemma}

\begin{proof}
Consider the principal ideal $\gamma\OO_K$ generated by a nonzero
element $\gamma$ of $I$. The ideal $I$ contains $\gamma \OO_K$, and
thus $|\OO_K/I| \le |\OO_K/\gamma\OO_K|$.  Because $N(I) = |\OO_K/I|$
and $|N(\gamma)| = |\OO_K/\gamma\OO_K|$, we have $|N(\gamma)| \ge
N(I)$, as desired.
\end{proof}

\begin{proof}[Proof of Theorem \ref{theorem:nf-coppersmith}]
As in the previous proofs, we will construct a polynomial $Q(x)$ in the
$\OO_K$-module generated by
\[
x^j f(x)^i I^{k-i} \quad \textup{for} \quad 0 \leq i < k \textup{ and } 0 \leq j < d
\]
and
\[
x^j f(x)^k \quad \textup{for} \quad 0 \leq j < t.
\]
Note that because of the ideals $I^{k-i}$, this is really a
pseudo-basis rather than a basis.

Let $m = dk+t$.  To represent this module, we will write down an $nm
\times nm$ matrix whose rows are a $\ZZ$-basis for a weighted version
of the module's canonical embedding.  Finding a short vector in this
lattice will correspond to finding a $Q$ that satisfies our bounds.

Our lattice is constructed much as before, except that in place of a
single entry for each coefficient of $x^j f(x)^i I^{k-i}$, we will have
an $n \times n$ block matrix.  Let $f_{sij}$ be the coefficient of
$x^s$ in $x^j f(x)^i$.  Then we form the ideal $f_{sij} I^{k-i}$, which
has an integral basis $b_1, \ldots, b_n$.  We incorporate the bounds
$\lambda_i$ on each absolute value into our canonical embedding for the
$s$-th coefficient of $x^j f(x)^i I^{k-i}$ by using
\[
\begin{pmatrix}
\lambda_1^s \sigma_1(b_1) & \lambda_2^s  \sigma_2(b_1) & \cdots & \lambda_n^s  \sigma_n(b_1) \\
\lambda_1^s  \sigma_1(b_2) & \ddots & \ & \lambda_n^s  \sigma_n(b_2)\\
\vdots && \ddots &  \vdots \\
\lambda_1^s  \sigma_1(b_n) & \lambda_2^s  \sigma_2(b_n) & \cdots & \lambda_n^s  \sigma_n(b_n)
\end{pmatrix}.
\]
This is equal to the product of the matrix with
$\lambda_1^s,\dots,\lambda_n^s$ on the diagonal with the canonical
embedding $\sigma(b)$, so the absolute value of the determinant of the
block is
\[
\lambda_1^s \dots \lambda_n^s \sqrt{|\Delta_K|}\,|N(f_{sij})|\,N(I)^{k-i}.
\]

Now consider a vector $v$ in this lattice and the polynomial $Q(x) =
\sum_j q_j x^j$ that it represents.  If $|w|_i < \lambda_i$ for all
$i$, then we can bound $|N(Q(w))|$ using the $\ell_1$ norm by applying
the arithmetic mean-geometric mean inequality. We have
\[
|N(Q(w))| = \prod_i \bigg| \sum_j q_{j} w^j \bigg|_i,
\]
and hence
\begin{align*}
|N(Q(w))|^{1/n} & \leq \frac{1}{n} \sum_i \bigg| \sum_j q_{j} w^j \bigg|_i \\
& \leq \frac{1}{n} \sum_i \sum_j |q_{j}|_i \lambda_i^j.
\end{align*}
Thus,
\[
|N(Q(w))| \leq \left( \frac{1}{n} |v|_1 \right)^n.
\]

As in the integer case, LLL produces a nonzero vector $v$ whose
$\ell_1$ norm is bounded by
\[
\sum_i \sum_j |v_i|_j \leq \sqrt{nm} 2^{(nm-1)/4} |\det(M)|^{\frac{1}{nm}}.
\]
Note that here, $|v_i|_j$ denotes the $j$-th number field norm applied
to the $i$-th entry of $v$.

Now it remains to compute the determinant of our weighted canonical
embedding. The lattice basis we produced in our construction is block
upper triangular, so the determinant is the product of the blocks on
the diagonal.  Letting $\prod_i \lambda_i = X$, we get
\begin{align*}
|\det M| & = \prod_{0 \leq i < m} \big(X^i \sqrt{|\Delta_K|}\big)
\prod_{0 \leq j \leq k}N(I)^{dj} \\
&= \sqrt{|\Delta_K|}^{\,m}
X^{{m(m-1)}/{2}}N(I)^{{dk(k+1)}/{2}}.
\end{align*}

Thus, we have
$$
|v|_1 < \sqrt{nm} 2^{(nm-1)/4} \sqrt{|\Delta_K|}^{\, \frac{1}{n}}
\left( X^{{m(m-1)}/{2}}N(I)^{{dk(k+1)}/{2}} \right)^{\frac{1}{nm}}.
$$
Recall that if $|w|_i < \lambda_i$ for all $i$, then
$$
|N(Q(w))| \le \frac{1}{n^n} |v|_1^n.
$$

We will compute a $c$ so that
\[
\left(\frac{1}{n^n} \left(\sqrt{nm} 2^{(nm-1)/4}\right)^n
\sqrt{|\Delta_K|}\right)^{\frac{2}{m-1}} < c.
\]
Then by the same analysis as in the proof of
Theorem~\ref{theorem:coppersmith}, we can prove the theorem with a
bound of
\[
\frac{1}{c} N(I)^{{\beta^2}/{d}-\varepsilon}
\]
on the product $\prod_i \lambda_i$. A simple asymptotic analysis shows
that we can take $c = (2+o(1))^{n^2/2}$ as $m \to \infty$. Thus, we
achieve a bound of
\[
{(2+o(1))^{-n^2/2}} N(I)^{\beta^2/d - \varepsilon}.
\]
As before, we can take $\varepsilon = 1/\log N(I)$ to achieve in fact
${(2+o(1))^{-n^2/2}} N(I)^{\beta^2/d}$.

Note that so far, everything runs in polynomial time, with no
exponential dependence on $n$.  Unfortunately, removing the factor of
${(2+o(1))^{-n^2/2}}$ is computationally expensive.  We can use the
same trick as in Theorem~\ref{theorem:coppersmith}.  In the canonical
embedding of $\OO_K$, the region we would like to cover is a box of
dimensions $2\lambda_1 \times \dots \times 2 \lambda_n$ (the factor of
$2$ comes from including positive and negative signs).  The proof so
far shows that we can deal with a box that is a factor of
$(2+o(1))^{-n/2}$ smaller in each coordinate.  We can cover the large
box with ${(2+o(1))^{n^2/2}}$ of the smaller ones and compute the
solutions in each smaller box in polynomial time, but the total running
time becomes exponential in $n^2$.
\end{proof}

\subsection{Solving the closest vector problem in ideal lattices}
\label{subsec:bdd}

In \cite{peikert}, Peikert and Rosen proposed using the closest vector
problem for ideal lattices as a hard problem for use in constructing
lattice-based cryptosystems.  In \cite{LPR}, Lyubashevsky, Peikert, and
Regev gave hardness reductions for such cryptosystems via the
bounded-distance decoding problem, defined for the $\ell_\infty$ norm
as follows.  Given an ideal $I$ in $\OO_K$, a distance $\delta$, and an
element $y \in K$, find $y+w \in I$ such that $|w|_\infty < \delta$,
where $|\cdot|_\infty$ denotes the $\ell_\infty$ norm on $K$ (i.e.,
the maximum of the $n$ absolute values).

If $y \in \OO_K$, then we can define $f(x) = x+y$ and find the roots
$w$ of $f(x) \equiv 0 \pmod{I}$ satisfying
\[
|w|_\infty < (2+o(1))^{-n/2} N(I)^{1/n}.
\]
This amounts to taking $d=1$, $\beta=1$, and $\lambda_1 = \dots =
\lambda_n = (2+o(1))^{-n/2} N(I)^{1/n}$. Because we are using the
$\ell_\infty$ norm, the minimal nonzero norm of $I$ is at most
$\big(\sqrt{|\Delta_K|} N(I)\big)^{1/n}$. Thus, our algorithm can
handle distances $\delta$ less than $(2+o(1))^{-n/2}
|\Delta_K|^{-1/(2n)}$ times the minimal norm of $I$. (Of course, this
is somewhat worse than using LLL directly.) Note also that if $y
\not\in \OO_K$, then we can rescale $y$ and $I$ by a positive integer
to reduce to the previous case.

If the $(2+o(1))^{-n^2/2}$ could be improved to
$2^{-n}\sqrt{|\Delta_K|}$, then we could solve the bounded-distance
decoding problem up to half the minimal distance, by the same argument
as above with $\lambda_1 = \dots = \lambda_n =  |\Delta_K|^{1/(2n)}
N(I)^{1/n}/2$.  This suggests that it will be difficult to remove the
multiplicative factor entirely.

\section{Function Fields}
\label{section:functionfield}

Much as number fields are finite extensions of $\QQ$, \emph{function
fields} are finite extensions of the field $\FF_q(x)$ of rational
functions over a finite field $\FF_q$.  They arise naturally from
algebraic curves over $\FF_q$, as the field of rational functions on
the curve. For example, for a plane curve defined by the polynomial
equation $f(x,y)=0$, the function field will be $\FF_q(x,y)/(f(x,y))$
(i.e., rational functions of $x$ and $y$, where the variables satisfy
$f(x,y)=0$).  See \cite{stichtenoth} and \cite{rosen} for background on
function fields, and \cite{lorenzini} for a beautiful account of the
analogies between number fields and function fields.

More generally, let $\XX$ be an algebraic curve over $\FF_q$.
Specifically, it must be a smooth, projective curve that remains
irreducible over the algebraic closure of $\FF_q$.  Our function field
$K$ will be the field of rational functions on $\XX$ defined over
$\FF_q$.  (Note that we are assuming $\FF_q$ is the full field of
constants in $K$; in other words, each element of $K$ is either in
$\FF_q$ or transcendental over $\FF_q$.)

Let $\XX(\FF_q)$ be the set of points on $\XX$ with coordinates in
$\FF_q$.  Every point $p \in \XX(\FF_q)$ gives a valuation $v_p$ on
$K$, which measures the order of vanishing at that point. Poles are
treated as zeros of negative order.  The corresponding absolute value
on $K$ is defined by
\[
|f|_p = q^{-v_p(f)}.
\]
(Note that this is not the $\ell_p$ norm on a vector; in this section, the
$\ell_p$ norm will not be used.) In other words, high-order zeros make a
function small, while poles make it large. Not every absolute value on $K$ is
of this form---there is a slight generalization that corresponds to points
defined over finite extensions of $\FF_q$ (more precisely, Galois orbits of
such points). For our purposes we can restrict our attention to the absolute
values defined above, but in fact all our results generalize naturally to
places of degree greater than $1$.

In the number field case, the Archimedean absolute values (which come from
the complex embeddings) play a special role, although there are infinitely
many non-Archimedean absolute values as well, namely the $p$-adic absolute
values measuring divisibility by primes. In the function field case, there
are no Archimedean absolute values, and any set of absolute values can play
that role.

Let $S$ be a nonempty subset of $\XX(\FF_q)$, and let $\OO_S$ be the
subring of $K$ consisting of all rational functions whose poles are
confined to the set $S$.  The ring $\OO_S$ is analogous to the ring of
algebraic integers in a number field; in this analogy, the condition of
having no poles outside $S$ amounts to the condition that an algebraic
integer has no primes in its denominator, because the valuations from
points outside $S$ correspond to the $p$-adic valuations.

For example, if $\XX$ is the projective line (i.e., the ordinary line
completed with a point at infinity), then $K$ is simply the field
$\FF_q(z)$ of rational functions in one variable.  If we let $S =
\{\infty\}$ be the set consisting solely of the point at infinity, then
$\OO_S$ is the set of rational functions that have poles only at
infinity.  In other words, it is the polynomial ring $\FF_q[z]$.  (A
polynomial of degree $d$ has a pole of order $d$ at infinity.)

The norm of an element $f \in \OO_S$ is defined by
\[
N(f) = \prod_{p \in S} |f|_p,
\]
and the norm of a nonzero ideal $I$ is defined by $N(I) = |\OO_S/I|$.
As in the number field case, the norm of the principal ideal $f \OO_S$
is $N(f)$.

\subsection{Background on algebraic-geometric codes}

Algebraic-geometric codes are a natural generalization of Reed-Solomon
codes.  They are of great importance in coding theory, because for
certain finite fields they beat the Gilbert-Varshamov bound (which is
the performance of a random code, and which aside from
algebraic-geometric codes is the best bound known).  See Section~8.4 in
\cite{stichtenoth}.

To define an algebraic-geometric code on $\XX$, we specify for each point in
$S$ the maximum allowable order of a pole there, and we allow no poles
outside of $S$. The space of functions satisfying these restrictions is a
finite-dimensional $\FF_q$-vector space, and we can produce an
error-correcting code by looking at the evaluations of these functions at a
fixed set of points (disjoint from $S$).

This is typically described using the language of algebraic geometry. A
\emph{divisor} $D$ on $\XX$ is a formal $\ZZ$-linear combination of finitely
many points on $\XX$; the support of $D$ is the set of points with nonzero
coefficients. (We will restrict our attention to divisors supported at points
in $\XX(\FF_q)$.)  The divisor $D$ is called \emph{effective}, denoted $D
\succeq 0$, if all its coefficients are nonnegative.  For every function $f
\in K^*$, the \emph{principal divisor} $(f)$ is the sum of the zeros and
poles of $f$, with their orders as coefficients.  (The identically zero
function does not define a principal divisor, since it has a zero of infinite
order at every point.)  The \emph{degree} $\deg(D)$ of $D$ is the sum of its
coefficients, and the degree of a principal divisor is always zero.

Given a divisor $D$, the Riemann-Roch space $\LL(D)$ is defined by
\[
\LL(D) = \{0\} \cup \{f \in K^* : (f) + D \succeq 0\}.
\]
In other words, if the coefficient of $p$ in $D$ is $k$, then $f$ can
have a pole of order at most $k$ at the point $p$. The space $\LL(D)$
is a finite-dimensional $\FF_q$-vector space, and the famous
Riemann-Roch theorem describes its dimension:
\[
\dim_{\FF_q} \LL(D) = \deg(D) - g + 1 + \dim_{\FF_q} \LL(W-D),
\]
where $g$ is a nonnegative integer called the \emph{genus} of the curve
and $W$ is a particular divisor called the \emph{canonical divisor}. It
follows that $\dim_{\FF_q} \LL(D) \ge \deg(D) - g + 1$, and equality
holds if $\deg(D) > 2g-2$.

To translate the definition of an algebraic-geometric code to this
language, let $D$ be the divisor with support in $S$ whose coefficients
specify the allowed order of a pole at each point, and let
$p_1,\dots,p_n$ be distinct points in $\XX(\FF_q)$ but not in $S$. Then
the corresponding algebraic-geometric code consists of the codewords
$(w(p_1),\dots,w(p_n))$ for $w \in \LL(D)$.

In the case of the projective line, let $S = \{\infty\}$, so $\OO_S =
\FF_q[z]$, and let $D = d \infty$.  Then $\LL(D)$ is the space of
polynomials in $\FF_q[z]$ of degree at most $d$.  Thus, this
construction yields Reed-Solomon codes as a special case.

Theorem~\ref{theorem:ff-coppersmith} corresponds to list decoding of
algebraic-geometric codes in much the same way as
Theorem~\ref{theorem:polycop} does for Reed-Solomon codes.  The
evaluation points $p_1,\dots,p_n$ correspond to prime ideals
$P_1,\dots,P_n$ in $\OO_S$, where $P_i$ consists of the functions
vanishing at $p_i$, and we can let $I$ be the product $P_1\dots P_n$.
If the received codeword is $(y_1,\dots,y_n) \in \FF_q^n$, then we
define the linear polynomial $f$ so that $f(x) \equiv x-y_i \pmod{P_i}$
for all $i$.  (The Chinese remainder theorem lets us solve this
interpolation problem.)  Thus, for $w \in \OO_S$, $f(w)$ is in the
ideal $P_i$ if and only if $w(p_i)=y_i$.  We have $N(I) = q^n$, and
$\gcd(f(w) \OO_S,I)$ is divisible by $P_i$ exactly when $w(p_i)=y_i$.
Therefore the inequality
\[
N(\gcd(f(w)\OO_S,I)) \ge N(I)^\beta
\]
simply means that $w(p_i)=y_i$ for at least $\beta n$ values of $i$.
Thus, Theorem~\ref{theorem:ff-coppersmith} solves the list decoding
problem.

\subsection{Proof of Theorem~\ref{theorem:ff-coppersmith}}

The first obstacle to proving Theorem~\ref{theorem:ff-coppersmith} is
identifying the right sort of lattice to consider.  For comparison, in the
number field case, we use the canonical embedding to reduce from
$\OO_K$-lattices to $\ZZ$-lattices, because $\ZZ$ is a principal ideal domain
and hence $\ZZ$-lattices are structurally simpler. In the function field
case, $\FF_q[z]$-lattices are the analogous structures, but $\FF_q[z]$ has
infinitely many embeddings as a subring of $\OO_S$, while $\ZZ$ has only one
embedding into $\OO_K$.  We must identify an embedding of a special sort,
namely one that treats all the absolute values from points in $S$
evenhandedly. Lemma~\ref{lemma:equalval} accomplishes this.

One we have identified a suitable embedding of $\FF_q[z]$ into $\OO_S$, we
are faced with two more difficulties.  The first is that we must consider
lattices with more general non-Archimedean norms than those studied in the
literature, because we must take into account all the absolute values from
$S$, and the known algorithms for basis reduction no longer apply. However,
we can prove the needed results in our more general framework
(Lemma~\ref{lemma:findallsmall}).

The final difficulty comes from attempting to control the zeros and poles of
functions in $K$.  In the simplest function field, namely the rational
function field $\FF_q(z)$, we can specify the (finitely many) zeros and poles
arbitrarily, subject to just one constraint, that the total order of all the
zeros must equal that of the poles.  For example, $z^2/(z-1)$ has a zero of
order two at $0$, a pole of order one at $1$, and a pole of order one at
$\infty$ (because the function grows linearly as $z$ becomes large).

In more complicated function fields, there are additional subtle constraints
on the zeros and poles, which interfere with our ability to construct
auxiliary functions in the proof (specifically, the placeholder $X$ that
measures the size of the desired solution of the equation).  We circumvent
this difficulty in Lemma~\ref{lemma:sap}, using a technique based on the
strong approximation theorem.  This allows us to control the behavior of a
function at all the points in $S$ except one, if we are willing to allowed
uncontrolled behavior at that single point. Furthermore, we can uniformly
bound the bad behavior at the uncontrolled point in terms of the genus of the
function field. This approach introduces error terms into our bounds, but
they are small enough that they disappear entirely in the final result.

We now turn to the details of the proof.  To identify an appropriate
embedding of $\FF_q[z]$ into $\OO_S$, we would like to choose $z \in \OO_S$
so that $|z|_p$ is independent of $p$ for $p \in S$.  In that case, the
absolute values $|\cdot|_p$ with $p \in S$ will all restrict to the same
absolute value on the ring $R = \FF_q[z]$, which we will denote $|\cdot|$.

When $|S|=1$, we can choose any nonconstant element $z$ of $\OO_S$.
When $|S|>1$, it is not as trivial, but fortunately there is always
such an element:

\begin{lemma} \label{lemma:equalval}
There exists an integer $a \ge 1$ and an element $z \in \OO_S$ such
that $v_p(z) = -a$ for all $p \in S$, and we can find such an element
in probabilistic polynomial time.
\end{lemma}

\begin{proof}
Let $\Delta_a$ be the divisor
\[
\sum_{p \in S} a p
\]
with coefficient $a$ for each $p \in S$, and let $g$ be the genus of
the curve $\XX$. If $a|S| > 2g-2$, then by the Riemann-Roch theorem,
\[
\dim_{\FF_q} \LL(\Delta_a) = a|S| - (g-1).
\]
Furthermore, if $a|S| > 2g-1$, then for each $p \in S$,
\[
\dim_{\FF_q} \LL(\Delta_a - p) = a|S| - g.
\]
Thus, if $|S| < q$, then $\LL(\Delta_a)$ cannot be contained in the
union of $\LL(\Delta_a - p)$ over all $p \in S$, and therefore there
exists a function with poles of order exactly $a$ at each point in $S$.
If $|S| < q/2$, then it is easy to find such a function by random
sampling, since at least half the elements in $\LL(\Delta_a)$ will
work. (Recall that as mentioned in Section~\ref{subsec:ffsintro}, we
assume that we can efficiently compute bases of Riemann-Roch spaces.)

This proof requires $|S| < q$, but the same idea works if we pass to a finite
extension $\FF_{q^i}$ of $\FF_q$, and $|S| < q^i$ then suffices. Thus,
if we take $i$ large enough, there exists a function defined over $\FF_{q^i}$
with poles of equal order $a$ at the points in $S$ (and no poles elsewhere).
Now multiplying the $i$ conjugates of this function over $\FF_q$ produces
such a function over $\FF_q$, as desired, with poles of order $a i$.  Taking
$q^i > 2|S|$ gives an efficient algorithm as well.
\end{proof}

For the rest of this section, let $z$ be such a function and let $R =
\FF_q[z]$.  Then the ring $\OO_S$ is a free $R$-module of rank $a|S|$
by Theorem~1.4.11 in \cite{stichtenoth}, as is every nonzero ideal in
$\OO_S$.

As in the previous proofs, we will construct a polynomial $Q(x)$ in the
$\OO_S$-module $\MM$ generated by
\[
x^j f(x)^i I^{k-i} \quad \textup{for} \quad 0 \leq i < k
\textup{ and } 0 \leq j < d
\]
and
\[
x^j f(x)^k \quad \textup{for} \quad 0 \leq j < t.
\]
Let $m = dk+t$.

The module $\MM$ is a submodule of the $\OO_S$-module $\PP$ of
polynomials of degree less than $m$, which is a free $\OO_S$-module of
rank $m$ and hence a free $R$-module of rank $ma|S|$.  Thus, as in the
setting of Lemmas~\ref{lemma:small} and~\ref{lemma:findallsmall}, we
are working with an $R$-module contained in a free $R$-module.

We want $Q(x)$ to have the property that for $w \in \LL(D)$,
\[
N(Q(w)) < N(I)^{\beta k}.
\]
In fact, we will bound $N(Q(w))$ by
\[
N(Q(w)) = \prod_{p \in S} |Q(w)|_p \le
\left(\max_{p \in S} |Q(w)|_p\right)^{|S|},
\]
and we will ensure that
\[
\left(\max_{p \in S} |Q(w)|_p\right)^{|S|} < N(I)^{\beta k}.
\]

Let $q_0,\dots,q_{m-1}$ denote the coefficients of $Q$, so
\[
Q(x) = \sum_{i=0}^{m-1} q_i x^i.
\]
Then
\[
|Q(w)|_p \le \max_i |q_i|_p |w|_p^i.
\]
Suppose the divisor $D$ is given by
\[
D = \sum_{p \in S} \lambda_p p.
\]
Then $|w|_p \le q^{\lambda_p}$ for $w \in \LL(D)$, and thus
\[
|Q(w)|_p \le \max_i |q_i|_p \, q^{i \lambda_p}.
\]

To emulate the analysis from Sections~\ref{section:coppersmith}
and~\ref{section:polycop}, we would like to find $X \in \OO_S$ such that
$v_p(X) = -\lambda_p$ for all $p \in S$.  However, such an element does not
always exist.  Instead, we will construct an element with the desired
valuations at all but one point in $S$.  This approach is a special case of
the strong approximation theorem (Theorem~1.6.5 in \cite{stichtenoth} or
Theorem~6.13 in \cite{rosen}), but as we need only a weaker conclusion and
must consider computational feasibility, we will give a direct proof along
the same lines as Lemma~\ref{lemma:equalval}.

\begin{lemma} \label{lemma:sap}
Suppose $q \ge 2|S|$.  Then for any point $p_0 \in S$ and each divisor
$\sum_{p \in S} \mu_p p$ satisfying $\sum_{p \in S} \mu_p \ge 0$ and
$\mu_{p_0}=0$, there exists an element $X \in \OO_S$ such that $v_p(X)
= -\mu_p$ for all $p \in S \setminus \{p_0\}$, and $v_{p_0}(X) = -2g$,
where $g$ is the genus of $\XX$. Furthermore, we can construct such an
$X$ in probabilistic polynomial time.
\end{lemma}

\begin{proof}
Let $\Delta = \sum_{p \in S} \mu_p p + 2g p_0$.  Then $\deg(\Delta) \ge
2g$, and it follows from Riemann-Roch that $\dim_{\FF_q} \LL(\Delta) =
\deg(\Delta)-(g-1)$ and that $\dim_{\FF_q} \LL(\Delta-p) = \dim_{\FF_q}
\LL(\Delta) - 1$ for all $p \in S$.  We are looking for an element $X$
in $\LL(\Delta)$ but not $\LL(\Delta-p)$ for any $p \in S$.  By
assumption we can construct these Riemann-Roch spaces, and because $|S|
\le q/2$ at least half the elements of $X$ will have the desired
property, so we can find one by random sampling.
\end{proof}

The assumption that $q \ge 2|S|$ will hold in most applications: most
algebraic-geometric codes use a small set $S$, and in fact $|S|$ cannot
be much larger than $q$ because $S \subseteq \XX(\FF_q)$ and
$|\XX(\FF_q)| \le q + 2g\sqrt{q} + 1$ (see Theorem~5.2.3 in
\cite{stichtenoth}). However, if $|S|
> q/2$, then we can simply pass to a finite extension of $\FF_q$.
Thus, without loss of generality we can assume that $q \ge 2|S|$.

By assumption in Theorem~\ref{theorem:ff-coppersmith}, the support of
$D$ is a proper subset of $S$, so we can let $p_0 \in S$ be a point
such that $\lambda_{p_0} = 0$.  Because of the limitations of the
strong approximation theorem, we require such a point to make the
remainder of the proof work. This is not an obstacle to the
applicability of the theorem, because algebraic-geometric codes will
generally not use every point in $\XX(\FF_q)$ for poles or evaluation
points, and if they do we can pass to a finite extension of $\FF_q$ to
generate more points.  Note also that we can assume $\deg(D) \ge 0$,
because otherwise $\LL(D)$ is the empty set.

Now, Lemma~\ref{lemma:sap} lets us construct an element $X \in \OO_S$
such that $v_p(X) = -\lambda_p$ for $p \in S \setminus \{p_0\}$. This
element has the property that $v_p(X^i) = -i\lambda_p$ for $p \in S
\setminus \{p_0\}$.  Unfortunately, the valuation at $p_0$ grows
linearly with $i$ as well, and that will damage our bounds. However, we
can avoid that problem by applying Lemma~\ref{lemma:sap} to construct
elements $X_i$ so that $v_p(X_i) = -i\lambda_p$ for $p \in S \setminus
\{p_0\}$ while maintaining $v_{p_0}(X_i) = -2g$.  Of course we set
$X_0=1$.

In terms of the elements $X_i$, we have
\[
|Q(w)|_p \le \max_i |q_i X_i|_p
\]
for $p \in S \setminus \{p_0\}$.  Furthermore, this inequality holds
for $p = p_0$ because $v_{p_0}(w) \ge 0 \ge v_{p_0}(X_i)$.

Define the norm of a polynomial $\sum_i c_i x^i \in \PP$ (with $c_i \in
\OO_S$) by
\[
\bigg| \sum_i c_i x^i \bigg| = \max_i \max_{p \in S} |c_i|_p.
\]
Note that this defines a non-Archimedean norm on the free $R$-module
$\PP$ satisfying all three properties required in
Section~\ref{subsec:findingshort} (with the absolute value $|\cdot|$ on
$R$).  Here, we crucially use the fact that we have only one absolute
value on $R$; if that were not the case, then property~3 would fail.

Let $T \colon \PP \to \PP$ be the linear transformation that multiplies
the degree $i$ term by $X_i$.  Then
\[
\max_{p \in S} |Q(w)|_p \le \max_{p \in S} \max_i |q_i X_i|_p = |T Q|.
\]
Thus, it will suffice to construct a nonzero polynomial $Q \in \MM$
such that $|TQ|^{|S|} < N(I)^{\beta k}$.

Now we can apply Lemma~\ref{lemma:findallsmall}.  We need to determine
two things: the geometric mean $C$ of the norms of an $R$-basis of
$\PP$ and the dimension of the quotient $\PP/T\MM$.  Then there exists
a nonzero $Q \in \MM$ such that
\[
|TQ| \le C |z|^{\dim_{\FF_q}(\PP/T\MM)/(a|S|m)} = C q^{\dim_{\FF_q}(\PP/T\MM)/(|S|m)},
\]
because these $R$-modules have rank $a|S|m$ and $|z|=q^a$.

Let $b_1,\dots,b_{a|S|}$ be any $R$-basis of $\OO_S$, and let
\[
C = \bigg( \prod_{i=1}^{a |S|} \max_{p \in S}
|b_i|_p \bigg)^\frac{1}{a |S|}.
\]
Then the elements $b_i x^j \in \PP$ (with $1 \le i \le a|S|$ and $0 \le
j < m$) form an $R$-basis of $\PP$, and the geometric mean of their
norms is $C$ because $|b_i x^j|$ is independent of the degree $j$.

To compute the dimension of $\PP/T\MM$, note that the generators of
$\MM$ are triangular (i.e., given by polynomials of each degree). Thus,
we merely need to add the dimensions of the quotients of $\OO_S$ by the
ideals of leading coefficients.  From the polynomials $X_{di+j} x^j
f(x)^i I^{k-i}$, we see that the leading coefficients form the ideal
$X_{di+j} I^{k-i}$.  Thus,
\begin{align*}
q^{\dim_{\FF_q} \PP/T\MM} &= |\PP/T\MM|\\
 &=
N(I)^{d k(k+1)/2} \prod_{i=0}^{m-1} N(X_i)\\
&=
N(I)^{d k(k+1)/2} \bigg( \prod_{p \in S} q^{\lambda_p m(m-1)/2}
\bigg) \prod_{i=0}^{m-1} |X_i|_{p_0}.
\end{align*}
In other words,
\[
q^{\dim_{\FF_q} \PP/T\MM} \le N(I)^{dk(k+1)/2} q^{\deg(D) m(m-1)/2} q^{2mg}.
\]

Now applying Lemma~\ref{lemma:findallsmall} shows that we can find a
nonzero polynomial $Q \in \MM$ such that
\[
|T Q|^{|S|} \le C q^{2g} q^{\deg(D) (m-1)/2} N(I)^{dk(k+1)/(2m)}.
\]
We want to achieve $|TQ|^{|S|} < N(I)^{\beta k}$.  Let $N(I) = q^n$ and
\[
\ell = \deg(D) + \frac{2}{m-1} \log_q \big(C q^{2g}\big).
\]
Then Lemma~\ref{lem:k-and-m} applies, and shows that we can achieve
$|TQ|^{|S|} < N(I)^{\beta k}$ whenever $\ell < n \left(
\frac{\beta^2}{d} - \varepsilon \right)$, which is equivalent to
\[
\left (C q^{2g} \right)^{\frac{2}{m-1}} q^{\deg(D)} <
N(I)^{\frac{\beta^2}{d}-\varepsilon}.
\]
We can take the denominator of $\beta$ to be a divisor of $n$ (because
$N(I)=q^n$).  Thus, $N(I)^{\beta^2/d}$ is an integral power of
$q^{1/(nd)}$, as of course is $q^{\deg(D)}$, and to prove the bound in
Theorem~\ref{theorem:ff-coppersmith} it suffices to prove it to within
a factor of less than $q^{1/(nd)}$.

Now let $\varepsilon < 1/(2n^2 d)$ and $m > 1 + 4n d (2g + \log_q C)$. Then
$N(I)^\varepsilon$ and $\left (C q^{2g} \right)^{\frac{2}{m-1}}$ are both
strictly less than $q^{1/(2nd)}$.  Thus, our algorithm works as long as
\[
q^{\deg(D)} < N(I)^{\beta^2/d}.
\]
This completes the proof of Theorem~\ref{theorem:ff-coppersmith}.

\section*{Acknowledgements}

We are grateful to Amanda Beeson, Keith Conrad, Abhinav Kumar, Victor Miller,
Vincent Neiger, Chris Peikert, Bjorn Poonen, Nigel Smart, and Madhu Sudan for
helpful conversations, comments, and references.

\end{document}